\documentclass{article}
\usepackage{amsmath,amsthm,amssymb,amscd}

\newcommand{\ga}{\alpha}
\newcommand{\gb}{\beta}

\newcommand{\gd}{\delta}
\newcommand{\gw}{\omega}

\newcommand{\gs}{\sigma}
\newcommand{\eps}{\varepsilon}

\newcommand{\cantor}{2^\gw}

\newcommand{\supp}{\mathrm{supp}}

\newcommand{\dom}{\mathrm{dom}}

\newcommand{\diam}{\mathrm{diam}}

\newcommand{\coll}{\mathrm{Coll}}

\newtheorem{theorem}{Theorem}[section]

\newtheorem{claim}[theorem]{Claim}
\newtheorem{corollary}[theorem]{Corollary}
\newtheorem{fact}[theorem]{Fact}
\newtheorem{proposition}[theorem]{Proposition}

\theoremstyle{definition}
\newtheorem{definition}[theorem]{Definition}
\newtheorem{example}[theorem]{Example}
\newtheorem{question}[theorem]{Question}
\title{Coloring equilateral triangles\footnote{2020 AMS subject classification 03E35, 14P99, 05C15.}}

\author{
Jind{\v r}ich Zapletal\\
University of Florida\\
zapletal@ufl.edu}

\begin{document}
\maketitle

\begin{abstract}
It is consistent relative to an inaccessible cardinal that ZF+DC holds, the hypergraph of equilateral triangles on a given Euclidean space has countable chromatic number, while the hypergraph of isosceles triangles on $\mathbb{R}^2$ does not.
\end{abstract}

\section{Introduction}

This paper deals with chromatic numbers of hypergraphs of arity three on Euclidean spaces. It has been known since the 1960's that the chromatic number of the hypergraph of equilateral triangles in $\mathbb{R}^2$ is countable in ZFC \cite{ceder:basis}. Erd\H os asked whether the same is true for the hypergraph of isosceles triangles in $\mathbb{R}^2$. On the way towards the solution of this problem, Schmerl first \cite{schmerl:triangle} gave a brief argument that the chromatic number of the hypergraph of equilateral triangles in any dimension is countable, later \cite{schmerl:partitions} gave an affirmative answer to the question of Erd\H os, and eventually classified all algebraic hypergraphs of countable chromatic number in a major breakthrough \cite{schmerl:avoidable}.

This paper is a commentary on this development of events. I will show that Schmerl's theorems appeared in order of increasing difficulty. Namely, I will prove

\begin{theorem}
\label{maintheorem}
If the theory ZFC+there is an inaccessible cardinal is consistent, then so is ZF+DC+for every number $d\geq 1$, the hypergraph of equilateral triangles in dimension $d$ has countable chromatic number+the hypergraph of isosceles triangles in $\mathbb{R}^2$ does not have countable chromatic number.
\end{theorem}

\noindent  Theorem~\ref{maintheorem} greatly understates the understanding of the resulting models of ZF+DC. They are obtained as balanced extensions of the choiceless Solovay model. Therefore, as a matter of general properties of such models, they do not contain e.g.\ maximal almost disjoint families \cite[Theorem 14.1.1]{z:geometric}. The posets generating them are $(4, 3)$-balanced (Corollary~\ref{balancecorollary}) and as a result, the models do not contain any discontinuous homomorphisms between Polish groups \cite[Theorem 13.2.1]{z:geometric}. In the model, every nonmeager (or non-null, in the sense of the usual two-dimensional Lebesgue measure) subset of $\mathbb{R}^2$ contains an isosceles triangle.  It is possible to achieve the independence result for all dimensions simultaneously. The posets are $d$-Noetherian balanced, which makes it possible to perform further dimension-specific analysis as in \cite{z:distance}. The inaccessible cardinal is needed only to initialize the methodology of geometric set theory independence proofs \cite[Part II]{z:geometric}; I doubt that it is necessary for the independence result.

There are many open questions in the area. The first one can be asked about any algebraic hypergraph without perfect cliques. A general negative answer is known only for arity two, i.e.\ algebraic graphs without perfect cliques \cite{z:ngraphs}. The case of equilateral triangles in dimension two is resolved in \cite{z:vitali}.

\begin{question}
Let $d\geq 3$ be a number. In ZF+DC, does countable chromatic number of hypergraph of equilateral triangles in $\mathbb{R}^d$ imply the existence of a Vitali set?
\end{question}

\noindent The second question deals with a specific type of coloring. Given a hypergraph $\Gamma$ on a topological space $X$, a $\Gamma$-\emph{free neighborhood assignment} is a map $c$ assigning to each point $x\in X$ an open neighborhood of $x$ in such a way that for every hyperedge $e\in\Gamma$ there is a vertex $x\in e$ such that $e\not\subseteq c(x)$.

\begin{question}
Let $d\geq 2$ be a number. Is there a balanced extension of the Solovay model in which there is an equilateral-triangle-free neighborhood assignment?
\end{question}

\noindent As for the architecture of the paper, Section~\ref{algebraicsection} collects the necessary concepts from real algebraic geometry, with the central Definition~\ref{centraldefinition}(2). Section~\ref{posetsection} constructs the coloring poset generating the model for Theorem~\ref{maintheorem}. Finally, Section~\ref{independencesection} shows how to prove that in the resulting model, there is no coloring of isosceles triangles in $\mathbb{R}^2$. The main point is very easy to describe: if $\{x_0, x_1, x_2\}$ is an isosceles triangle in $\mathbb{R}^2$ which is in a suitable sense generic, then $V[x_0, x_1]\cap V[x_1, x_2]\cap V[x_0, x_2]=V$. On the other hand, if this is an equilateral triangle in any dimension, then the intersection of the three models contains the length of the side of the triangle. The whole argument turns on this simple geometric point. The terminology of the paper sticks to the set theoretic standard of \cite{jech:newset}. For a bounded nonempty subset $O$ of a Euclidean space, I write $\diam(O)$ for the supremum of all Euclidean distances between the points in the set $O$. 

\section{Algebraic geometry}
\label{algebraicsection}

The proof uses several standard concepts from real algebraic geometry. To begin, recall \cite[Section 3.3]{marker:book} that the language of real closed fields includes symbols for addition, multiplication and inequality. The theory of real closed fields includes the axioms of an ordered field, the statement that squares are exactly the non-negative elements, and the statements saying that every nontrivial polynomial of odd degree has a root. The theory of real closed fields admits elimination of quantifiers \cite[Theorem 3.3.15]{marker:book}. For algebraic sets, I will need the following.

\begin{definition}
Let $n\geq 1$ be a number, and let $A\subseteq\mathbb{R}^d$ be a set.

\begin{enumerate}
\item The set $A$ is \emph{algebraic} if it is of the form $\{z\in\mathbb{R}^d\colon p(z)=0\}$ where $p$ is a polynomial with real coefficients;
\item the set $A$ is \emph{irreducible} if $A\subset\mathbb{R}^d$ is algebraic and whenever $A=B\cup C$ is a union of two algebraic sets, then either $B=A$ or $C=A$;
\item \cite[Definition 3.3.3]{bochnak:real} a point $x\in A$ is \emph{nonsingular} in $A$ if the dimension of the Zariski tangent space to $A$ at $x$ has the smallest possible value $\dim(A)$, where the dimension is as defined in \cite[Definition 2.8.1]{bochnak:real}.
\end{enumerate}
\end{definition}

\noindent The following fact records the basic results about algebraic sets and their nonsingular points.

\begin{fact}
\label{algebraicfact}
Let $n\geq 1$ be a number.

\begin{enumerate}
\item \textnormal{(Hilbert basis theorem)} There is no infinite sequence of algebraic sets which is strictly decreasing with respect to inclusion;
\item \textnormal{\cite[Theorem 2.8.3]{bochnak:real}} every algebraic set $A\subseteq\mathbb{R}^d$ has a unique decomposition into irreducible components: $A=\bigcup\{A_i\colon i\in j\}$ where $j\geq 1$ is a number and each sets $A_i$ is algebraic, irreducible, and not included in the union of the others; 
\item \textnormal{\cite[Claim 4.9]{z:distance}} if $A, B\subseteq\mathbb{R}^d$ are algebraic sets and $A$ is irreducible, then either $A\subseteq B$ or $A\cap B$ is nowhere dense in the set of nonsingular points of $A$.
\end{enumerate}
\end{fact}

\noindent Coming to the original contents of this paper, the following concept is key.

\begin{definition}
\label{centraldefinition}
Let $n\geq 1$ be a number, $x\in\mathbb{R}^d$, and $F\subseteq\mathbb{R}$. 

\begin{enumerate}
\item A set $A\subset\mathbb{R}^d$ is \emph{$F$-semialgebraic} if it is of the form $\{z\in\mathbb{R}^d\colon \mathbb{R}\models \psi(z)\}$ where $\psi$ is a a formula of the language of real closed fields with parameters in $F$;
\item $D(F, x)$ is the smallest set among the collection of $F$-semialgebraic sets containing $x$, which happen to be algebraic.
\end{enumerate}
\end{definition}

\begin{example}
Let $d=1$, $F=\mathbb{Q}$, and $x=\sqrt{2}$. Then $D(F, x)=\{\sqrt{2}\}$, even though the smallest zero set of a polynomial with rational coefficients containing $x$ is $\{-\sqrt{2}, \sqrt{2}\}$.
\end{example}

\noindent If the set $F\subset\mathbb{R}$ is a real closed subfield, then the set $D(F,x)$ is simply the smallest subset of $\mathbb{R}^d$
 which is a zero-set of a polynomial with coefficients in $F$ and contains $x$. This follows from $F$ being an elementary submodel of $\mathbb{R}$. However, we will use the sets $D(F, x)$ also when $F$ is a real closed subfield of $\mathbb{R}$ with finitely many additional elements, and that is where its exact definition will pay off. The following three propositions provide the most important information about the sets $D(F, x)$.

\begin{proposition}
\label{elementaryproposition}
Let $F\subset\mathbb{R}$ be a set and $x\in\mathbb{R}^d$ be a point. Then

\begin{enumerate}
\item $D(F, x)$ is an irreducible algebraic set;
\item $x$ is a nonsingular element of $D(F, x)$;
\item if $F\subset\mathbb{R}$ is a real closed subfield, then $D(F, x)=\{x\}$ if and only if $x\in F^d$.
\end{enumerate}
\end{proposition}

\begin{proof}
(1) follows from the fact that the components of an algebraic set are definable from the algebraic set. To fill in the details, let $A\subset\mathbb{R}^d$ be an algebraic set. Let $p_i$ for $i\in j$ be polynomials with real coefficients whose zero sets are the components of the set $A$. Using dummy zero coefficients if necessary, we may assume that there is a number $m$ such that the sequence of coefficients of $p_i$ is an element of $\mathbb{R}^m$ for every $i\in j$. For each $i\in j$ let $O_i\subset\mathbb{R}^m$ be a basic open set containing the tuple of coefficients of $p_i$, so that the sets $O_i$ for $i\in j$ are pairwise disjoint. Now, the zero set of (say) the polynomial $p_0$ is defined as the set of all $x\in A$ such that there is a sequence $\langle q_i\colon i\in j\rangle$ of polynomials whose coefficients come from the set $O_i$ respectively, each of whose zero sets is a proper subset of $A$, and such that $q_0(x)=0$. This follows from the uniqueness of components of algebraic sets.

(2) follows from the fact that the set of singular points of an irreducible algebraic set $A$ is an algebraic, proper subset of it definable from $A$. The algebraicity and proper inclusion is proved in \cite[Proposition 3.3.14]{bochnak:real}. For the definability, it is enough to restate the definitions. Namely, for every polynomial $p$ such that $A=\{y\in\mathbb{R}^d\colon p(y)=0\}$, the Zariski tangent space to $A$ at $y$ is the set of all vectors perpendicular to the vector of partial derivatives of $p$ at $y$. The tangent space does not depend on the choice of $p$. The point $y$ is then nonsingular in $A$ if this space has maximal dimension possible, namely $\dim(A)$. This can be mechanically restated by a rather long formula in the language of real closed fields. 

(3) follows from the fact that a real closed subfield is an elementary submodel of $\mathbb{R}$. If $D(F, x)=\{x\}$ then each coordinate of $x$ is definable from parameters in $F$ and therefore an element of $F$.
\end{proof}

\begin{proposition}
\label{reflectionproposition}
Let $F_0, F_1$ be real closed subfields of $\mathbb{R}$ and let $n\in\gw$. For every finite set $a\subset\mathbb{R}$ there is a finite set $b\subset F_1$ such that for every $x\in F_1^n$, $D(F_0\cup b, x)\subseteq D(F_0\cup a, x)$.
\end{proposition}

\begin{proof}
Let $v\in\mathbb{R}^m$ be a tuple enumerating all elements of $a$. For each $x\in F_1^n$, let $p_x$ be a polynomial with $n+m$ free variables and coefficients in the field $F_0$ such that $p_x(x, v)=0$ and the set $A_x=\{u\in\mathbb{R}^n\colon p_x(u, v)=0\}$ is as small as possible. Such a polynomial exists by the Hilbert basis theorem. Let $B_x=\{w\in\mathbb{R}^m\colon p_x(x, w)=0\}$; this is an algebraic set containing $v$ as an element. Let $C=\bigcap\{B_x\colon x\in F_1^n\}$; this is an algebraic set containing $v$ as an element. Let $c\subset F_1^n$ be a finite set such that $C=\bigcap\{B_x\colon x\in c\}$. Let $b\subset F_1$ be a finite set such that all coordinates of points in $c$ can be found in the set $b$. I claim that the set $b$ works.

To show this, suppose that $x\in F_1^n$ is an arbitrary point, and work to show that $D(F_0\cup b, x)\subseteq D(F_0\cup a, x)$ holds. Use quantifier elimination for real closed fields to define the set $D(F_0\cup a, x)$ by a quantifier-free formula with parameters in $F_0\cup a$. There is a finite list of polynomials $q_i\colon i\in j$, each having parameters in $F_0$ and $n+m$ free variables such that the quantifier-free formula is a Boolean combination of formulas of the form $q_i(u, v)\geq 0$. Let $d=\{i\in j\colon q_i(x, v)=0\}$, let $D_x=\{u\in\mathbb{R}^n\colon q_i(u, v)=0$ for all $i\in n\}$, and let $O\subset\mathbb{R}^n$ be a basic open neighborhood of $x$ in which the polynomials $q_i(u, v)$ for $i\in j\setminus d$ do not change sign. Clearly, $A_x\subseteq D_x$ holds by the minimal choice of the polynomial $p_x$, and $D_x\cap O\subset D(F_0\cup a, x)$ holds by definitions.

Consider the set $E_x=\{u\in\mathbb{R}^d\colon\forall w\in C\ p_x(u, w)=0\}$. This is an algebraic set as well as a $F_0\cup b$-semi-algebraic set containing $x$ as an element. By its definition, $E_x\subseteq A_x$ holds. Thus, $E_x\cap O\subset D(F_0\cup a, x)$ must hold as well. Finally, since $D(F_0\cup b, x)\subseteq E_x$ holds by the definition of $D(F_0\cup b, x)$, it must be the case that $D(F_0\cup b, x)\cap O\subseteq D(F_0\cup a, x)$.

Now, consider the disjunction in Fact~\ref{algebraicfact}(3) applied to $A=D(F_0\cup b, x)$ and $B=D(F_0\cup a, x)$. It is impossible for the set $A\cap B$ to be nowhere dense in the set of nonsingular points of $A$, because $x\in O$ is a nonsingular point of $A$ by Proposition~\ref{elementaryproposition} and $B$ contains the whole neighborhood $A\cap O$ of $x$ in $A$. Thus, it must be the case that $A\subseteq B$ holds as desired.
\end{proof}

\noindent The final proposition of this section provides the key connection between real algebraic geometry and forcing.

\begin{proposition}
\label{productproposition}
Let $V[G_0], V[G_1]$ be mutually generic extensions and $n\geq 1$ be a number. Let $a\subset\mathbb{R}$ be a finite set in the model $V[G_1]$. Then for every $x\in\mathbb{R}^n\cap V[G_1]$, $D((\mathbb{R}\cap V)\cup a, x)=D((\mathbb{R}\cap V[G_0])\cup a, x)$.
\end{proposition}

\begin{proof}
Let $v\in\mathbb{R}^m$ be a finite sequence enumerating all elements of $a$. It follows directly from the definitions that $D((\mathbb{R}\cap V)\cup a, x)$ is just the $v$-th section of the set $D(\mathbb{R}\cap V, v^\smallfrown x)$, and similarly for $D((\mathbb{R}\cap V[G_0])\cup a, x)$. However, the sets $D(\mathbb{R}\cap V, v^\smallfrown x)$ and $D(\mathbb{R}\cap V[G_0], v^\smallfrown x)$ are equal by \cite[Corollary 2.7]{z:distance} applied to the Noetherian topology of algebraic sets.
\end{proof}

\section{The coloring poset}
\label{posetsection}

The model for Theorem~\ref{maintheorem} is obtained as a generic extension of the choiceless Solovay model \cite[Theorem 26.14]{jech:newset} via a definable $\gs$-closed coloring poset. The coloring poset can be stated for a more general class of hypergraphs of arity three.

\begin{definition}
A hypergraph $\Gamma$ on a Polish space $X$ of arity $n$ is \emph{$2$-irreflexive} if in the closure of $\Gamma$ in the space $[X]^{\leq n}$ with the Vietoris topology there are no sets of cardinality two.
\end{definition}

\noindent An example of a hypergraph which is $2$-irreflexive is the hypergraph of equilateral triangles on $\mathbb{R}^d$ if $d\geq 2$ is a number. An example of a hypergraph of arity three which is not $2$-irreflexive is the hypergraph of isosceles triangles on $\mathbb{R}^d$. $2$-irreflexive hypergraphs carry the following parameter:

\begin{definition}
\label{epsdefinition}
Let $X$ be a Euclidean space and $\Gamma$ be a $2$-irreflexive algebraic hypergraph on $X$. Let $A\subset\mathbb{R}$ be a set and $x\in X$ be a point. $\eps(A, x)$ is the largest number $\eps>0$ such that there is no set $b\subset X$ consisting only of points closer than $\eps$ to $x$ such that for every $y\in D(A, x)$ which is not equal to $x$ or to any element of $b$  $b\cup\{y\}\in\Gamma$ holds. If no such number $\eps$ exists, we put $\eps(A, x)=0$; if there are arbitrarily large such numbers $\eps$, put $\eps(A, x)=\infty$.
\end{definition}

\begin{proposition}
\label{epsproposition}
Let $\Gamma$ be an algebraic, $2$-irreflexive hypergraph of arity three on a Euclidean space $X$ of dimension $d\geq 1$.

\begin{enumerate}
\item if $A\subset\mathbb{R}$ then $\eps(A, x)=0$ if and only if $x$ is algebraic over $A$;
\item if $A\subset B\subset\mathbb{R}$ then $\eps(B, x)\leq\eps(A, x)$;
\item if $F\subset\mathbb{R}$ is a real closed subfield and $x\notin F^d$ is a point, then there is no set $b\subset F^d$ consisting of points closer than $\eps(F, x)$ to $x$ such that $b\cup\{x\}\in\Gamma$.
\end{enumerate}
\end{proposition}

\begin{proof}
For (1), recall that $x$ is algebraic over $A$ if and only if $D(A, x)=\{x\}$ by Proposition~\ref{elementaryproposition}(3). Now, if $D(A, x)=\{x\}$ then $\eps(A, x)=0$ by the definitions. If $D(A, x)\neq\{x\}$ then pick a point $y\in D(A, x)$ distinct from $x$. If for every $\eps>0$ there is a set $b_\eps$ consisting of points closer to $x$ than $\eps$ such that $b_\eps\cup\{y\}\in\Gamma$, then $\{x, y\}$ belongs to the closure of $\Gamma$ in the space $X^{\leq 3}$, contradicting the assumption of $2$-irreflexivity on $\Gamma$. Thus, $\eps(A, x)>0$ holds as desired.

(2) follows directly from the definitions, noting that $D(B, x)\subseteq D(A, x)$. For (3), suppose towards a contradiction that $b\subset F^d$ is a set consisting of points closer than $\eps(F, x)$ to $x$ such that $b\cup\{x\}\in\Gamma$. Let $B=\{y\in X\colon b\cup\{y\}\in\Gamma\}\cup b$. This is an $F$-algebraic set containing $x$, and therefore $D(F, x)\subseteq B$ holds. It follows that for every point $y\in D(F, x)\setminus (b\cup\{x\})$, $b\cup\{y\}\in\Gamma$ holds, and by the definitions, the set $b$ cannot consist of points which are closer than $\eps(F, x)$ to $x$.
\end{proof}

For the rest of the section, fix a number $d\geq 1$ and an algebraic hypergraph $\Gamma$ of arity three on $\mathbb{R}^d$ which is $2$-irreflexive. I will produce a balanced coloring poset for $\Gamma$. It will be useful to fix some initial data. The set of colors is the set of metric open balls in $\mathbb{R}^d$ which have a rational center and a rational radius. By a (partial) coloring I mean a (partial) map $f$ from $\mathbb{R}^d$ to the set of colors such that for each $x\in\dom(f)$, $x\in f(x)$ holds. Let $I$ be the ideal on the set $\mathbb{Q}^+$ which consists of those sets whose only accumulation point is zero.

\begin{definition}
\label{posetdefinition}
The coloring poset $P$ consists of all countable partial colorings $p$ such that for some countable real closed subfield $\supp(p)\subset\mathbb{R}$ it is the case that $\dom(p)=\supp(p)^n$. The ordering is given by $q\leq p$ if 

\begin{enumerate}
\item $p\subseteq q$;
\item for every point $x\in\dom(q)\setminus\dom(p)$, the metric diameter of $q(x)$ is smaller than $\eps (\supp(p), x)$;
\item for every finite set $a\subset\supp(q)$, writing $C(p, q, a)=\{x\in\dom(q\setminus p)\colon \eps(\supp(p)\cup a, x)\leq\diam(q(x))\}$ and $q^*A=\{\diam(q(x))\colon x\in A\}$, then $q^*C(p, q, a)\in I$ holds.
\end{enumerate}
\end{definition}

\noindent The properties of $P$ must be checked in turn.

\begin{proposition}
$\leq$ is a transitive relation.
\end{proposition}

\begin{proof}
Let $r\leq q\leq p$ be conditions in the coloring poset $P$; I must show that $r\leq p$ holds. For Definition~\ref{posetdefinition}(1), since $p\subseteq q$ and $q\subseteq r$ holds, $p\subseteq r$ certainly follows. For Definition~\ref{posetdefinition}(2), suppose that $x\in\dom(r)\setminus\dom(p)$ is a point. If $x\in\dom(q)$, then Definition~\ref{posetdefinition}(2) applied to $q\leq p$ implies that the diameter of $r(x)=q(x)$ is smaller than $\eps(\supp(p), x)$. If $x\in\dom(r)\setminus\dom(q)$, then Definition~\ref{posetdefinition}(2) applied to $r\leq q$ implies that $\diam(r(x))<\eps(\supp(q), x)\leq\eps(\supp(p), x)$ where the second inequality follows from Proposition~\ref{epsproposition}(2).

The key point is the verification of Definition~\ref{posetdefinition}(3). Suppose that $a\subset\supp(r)$ is a finite set. By Proposition~\ref{reflectionproposition}, find a finite set $b\subset\supp(q)$ such that for every point $x\in\dom(q)$, $D(\supp(p)\cup b, x)\subseteq D(\supp(p)\cup a, x)$. It is then clear that the set $C(p, r, a)$ is a subset of $C(p, q, b)\cup C(q, r, a)$. Now, $q^*C(p, q, b)\in I$ and $r^*C(q, r, a)\in I$ holds by Definition~\ref{posetdefinition}(3) applied to $q\leq p$ and to $r\leq q$. Thus, $r^*C(p,r,a)\in I$ and the proof is complete.
\end{proof}

\begin{proposition}
\label{densityproposition}
For every condition $p\in P$ and every point $x\in \mathbb{R}^d$ there is a condition $q\leq p$ such that $x\in\dom(q)$.
\end{proposition}

\begin{proof}
Fix $p\in P$ and $x\in\mathbb{R}^d$. Let $G\subset\mathbb{R}$ be any countable real closed subfield of $\mathbb{R}$ containing $\supp(p)$ as a subset and coordinates of the point $x$ as elements. Choose an infinite set $b$ of positive rationals converging to zero. Let $q$ be a function with domain $G^d$ which extends $p$ such that for every $y\in G^d\setminus\dom(p)$ $q(y)$ is a basic open neighborhood of $y$ of diameter which is in the set $b$ and smaller than $\eps(\supp(p), y)$ and such that $q\restriction G^d\setminus\dom(p)$ is an injection. It will be enough to show that $q\in P$ is a condition stronger than $p$.

To show that $q$ is a coloring, suppose that $e$ is a hyperedge in $\dom(q)$, and work to show that $e$ is not monochromatic. If all its vertices belong to $p$, then this follows from $p$ being a coloring. If more than one of its vertices belongs to $\dom(q\setminus p)$ then this follows from $q$ being an injection on this set. Finally, in the event that $e$ has exactly one vertex $y$ in the set $\dom(q\setminus p)$, apply Proposition~\ref{epsproposition}(3) to show that $e$ is not monochromatic in this case either.

Now, to show that $q\leq p$, it is clear that $p\subseteq q$ holds. For every point $y\in\dom(q\setminus p)$, the value $q(y)$ is a basic open neighborhood whose diameter is smaller than $\eps(\supp(y), y)$. This verifies Definition~\ref{posetdefinition}(2) for $p$ and $q$. Finally, for any finite set $a\subset\supp(q)$ the set $q^*C(p, q, a)$  belongs to the ideal $I$, because it is a subset of $q^*\dom(q\setminus p)$ which is in turn a subset of $b$. This verifies Definition~\ref{posetdefinition}(3) for $p$ and $q$ and completes the proof. 
\end{proof}

\begin{proposition}
\label{sigmaclosureproposition}
$\leq$ is $\gs$-closed.
\end{proposition}

\begin{proof}
If $\langle p_i\colon i\in\gw\rangle$ is a descending sequence of conditions then $p=\bigcup_ip_i$ is their common lower bound. To see this, first note that $\supp(p)=\bigcup_i\supp(p_i)$ is a real closed subfield of $\mathbb{R}$, $\dom(p)=\supp(p)^d$, and $p$ is a function and a coloring, therefore a condition in the poset $P$. Now, fix $i\in\gw$ and work to show that $p\leq p_i$ holds. Items (1) and (2) of Definition~\ref{posetdefinition} are immediate. For item (3), if $a\subset\supp(p)$ is a finite set, then there is $j\geq i$ such that $a\subset\supp(p_j)$, and then $C(p_i, p, a)=C(p_i, p_j, a)$ and (3) follows from the same item of the definition applied to $p_j\leq p_i$.
\end{proof}

\noindent Finally, we come to the balance properties of $P$. The moral of the following propositions is that suitable total $\Gamma$-colorings exist, they serve as balanced conditions in the poset $P$, and suitable strengthenings of balance actually occur. For the sake of brevity, the propositions are proved under the assumption of the Continuum Hypothesis (CH); this does not cause any injury to the consistency result I am aiming at. A significantly longer argument proves Proposition~\ref{barproposition} and then the others in ZFC.

\begin{proposition}
\label{barproposition}
(CH) For every condition $p\in P$ there is a total $\Gamma$-coloring $\bar p$ such that $\coll(\gw, \mathbb{R})\Vdash\bar p\leq p$.
\end{proposition}

\noindent Note that the total coloring $\bar p$ becomes a condition in $P$ in the $\coll(\gw, \mathbb{R})$-extension. The proposition says in particular that every countable partial coloring can be extended to a total one; however, it says more than that in view of Definition~\ref{posetdefinition}(3).

\begin{proof}
Use the Continuum Hypothesis assumption to find an enumeration $\langle r_\ga\colon\ga\in\gw_1\rangle$ of all reals. By transfinite recursion on the countable ordinal $\ga$ build conditions $p_\ga\in P$ so that $p_0=p$, $\ga\in\gb$ implies $p_\gb\leq p_\ga$, and $r_\ga\in\supp(p_{\ga+1})$. This is easy to do using Proposition~\ref{densityproposition} in successor stages, and Proposition~\ref{sigmaclosureproposition} in limit stages. Finally, let $\bar p=\bigcup_\ga p_\ga$. This is a total $\Gamma$-coloring. The proof of Proposition~\ref{sigmaclosureproposition} in the $\coll(\gw, \mathbb{R})$-extension shows that  $\coll(\gw, \mathbb{R})\Vdash\bar p$ is a common lower bound of all $p_\ga$ for $\ga\in\gw_1^V$ as desired.
\end{proof}

\begin{proposition}
\label{balanceproposition}
Let $\bar p\in V$ be a total coloring. Let $V[G_i]$ for $i\in k$ be triple-wise mutually generic extensions.  For each $i\in k$ let $p_i\leq \bar p$ be conditions in the model $V[G_i]$. Then $p_i$ for $i\in k$ have a common lower bound in the coloring poset $P$.
\end{proposition}

\noindent It is necessary to parse the proposition correctly. The whole proposition (and its proof as well) takes place in an ambient  forcing extension, whose choice is inconsequential, and it is therefore not specified. The assumption on the generic extensions says that any pair or any triple of them is mutually generic. In fact, the proof shows immediately that it is enough to assume that pairs and triples of them are $n$-Noetherian as in \cite{z:distance}--this is irrelevant for the purposes of this paper though. The coloring poset $P$ is re-interpreted in every model in question. The coloring $p$ is total in $V$, therefore uncountable in $V$. However, if a generic extension collapses the cardinality of $\mathbb{R}\cap V$ to $\aleph_0$, $p$ becomes a condition in the poset $P$ as interpreted in that generic extension. 

\begin{proof}
The first observation must be the following:

\begin{claim}
\label{coloringclaim}
$\bigcup_ip_i$ is a function and a coloring.
\end{claim}

\begin{proof}
Let $i, j\in k$ be distinct indices. Since the generic extensions $V[G_i]$ and $V[G_j]$ are mutually generic, $V[G_i]\cap V[G_j]=V$ holds. Thus, $\dom(p_i)\cap\dom(p_j)\subset V$. However, $p_i\restriction V=p_j\restriction V=p$ holds, so it must be the case that $\bigcup_{i\in k}p_i$ is a function. To see that it is a coloring, let $e=\{x_0, x_1, x_2\}$ be a hyperedge in $\dom(\bigcup_ip_i)$. The proof that it is not monochromatic considers several configurations:

\noindent\textbf{Case 1.} All three vertices are in the same extension $V[G_i]$. In this case, the hyperedge $e$ is not monochromatic because $p_i$ is a coloring.

\noindent\textbf{Case 2.} Not Case 1 and all three vertices can be found in the union of two of the forcing extensions. Then, there must be a single vertex, say $x_1$ which is in $V[G_i]\setminus V$, and the two remaining vertices are in $V[G_j]$ for some distinct indices $i, j\in k$. Let $D=\{y\in\mathbb{R}^n\colon \{x_0, x_2, y\}\in\Gamma\}$. This is a $V[G_j]$-algebraic set containing $x_1$. By the mutual genericity assumption on $V[G_i]$ and $V[G_j]$ and Proposition~\ref{productproposition}, $D(\mathbb{R}\cap V, x_1)\subseteq D$ holds; in other words, for every point $y\in D(\mathbb{R}\cap V, x_1)\setminus \{x_0, x_1, x_2\}$, $\{x_0, x_2, y\}\in\Gamma$ holds. By Definition~\ref{epsdefinition}, at least one of the points $x_0, x_2$ must be farther than $\eps(\mathbb{R}\cap V, x_1)$ from $x_1$. By Definition~\ref{posetdefinition}(2) applied to $p_1\leq p$, $\{x_0, x_2\}\not\subset p_1(x_1)$ and $e$ is not monochromatic.

\noindent\textbf{Case 3.} Not Case 1 and 2. This is handled in the same way as Case 2, except this time mutual genericity in triples is used. This proves the claim.
\end{proof}

\noindent It is important to understand that the work does not end here. While the function $\bigcup_{i\in k}p_i$ is a coloring, its domain is not in the required form for a condition in the coloring poset; it is necessary to extend its domain. Towards this end, I need to analyze the situation closer.

\begin{claim}
For every point $x\in\mathbb{R}^d\setminus\bigcup _i\dom(p_i)$  and all distinct indices $i, j\in k$ there is a set $b_{ij}(x)\in I$ such that whenever $y\in\dom(p_i\setminus p)$ and $z\in\dom(p_j\setminus p)$ such that $\{x, y, z\}\in\Gamma$, then either $\{x, z\}\not\subset p_i(y)$ or the diameter of $p_i(y)$ belongs to the set $b_{ij}(x)$.
\end{claim}

\begin{proof}
Fix the point $x$ and distinct indices $i, j\in k$. Let $c$ be the set of all pairs $\langle y, z\rangle$ such that $y\in \dom(p_i\setminus p)$ such that there is a point $z\in \dom(p_j\setminus p)$ such that $\{x, y, z\}\in\Gamma$ and $p_i(y)$ contains both $x$ and $z$.

For each pair $\langle y, z\rangle\in c$ let $A_{y,z}=\{w\in\mathbb{R}^d\colon \{y, z, w\}\in\Gamma\}\cup \{y, z\}$. This is an algebraic set. By the Hilbert basis theorem, there is a finite set $c'\subset c$ such that the intersections $\bigcap\{A_{y,z}\colon \langle y, z\rangle\in c\}$ and $\bigcap\{A_{y, z}\colon \langle y, z\rangle\in c'\}$ are equal; call their common value $A$. Let $a_i\subset\supp(p_i)$ be a finite set such that for every point $\langle y, z\rangle\in c'$, all coordinates of the point $y$ are in the set $a_i$; the set $a_j$ is defined similarly. Let $b_{ij}$ be the set of all diameters of values $p_i(y)$ such that $y\in\dom(p_i\setminus p)$ is a point such that $\eps(\supp(p_i)\cup a_i, y)$ is smaller than the diameter of $p_i(y)$. By Definition~\ref{posetdefinition}(3), this is a set in the ideal $I$. It will be enough to show that the set $b_{ij}$ works.

To see this, for every pair $\langle y, z\rangle\in C$, the set $D_{y,z}=\{w\in\mathbb{R}^n\colon\forall v\in c\ \{w, v, z\}\in\Gamma\}$ is an algebraic set which is defined as a semi-algebraic set in parameters in $a_i\cup a_j\cup z$. By Proposition~\ref{productproposition}, $D(\supp(p)\cup a_i, y)\subseteq D_{y,z}$. It follows that either the diameter of $p_i(y)$ is smaller than $\eps(\supp(p_i\cup a_i, y))$, in which case $\{x, z\}\not\subseteq p_i(y)$, or else the diameter of the value $p_i(y)$ belongs to the set $b_{ij}$ by the choice of $b_{ij}$.
\end{proof}

\noindent Now, let $F$ be any countable real closed subfield of $\mathbb{R}$ such that $\bigcup_{i\in k}\supp(p_i)\subset F$. Let $G=F^d\setminus\bigcup_{i\in k}\dom(p_i)$. Let $b\subset\mathbb{Q}^+$ be an infinite sequence converging to zero which has finite intersection with every set $b_{ij}(x)$ for distinct indices $i, j\in k$ and points $x\in G$. It is not difficult to find a map $q$ on $F^d$ extending $\bigcup_{i\in k}p_i$, such that $q\restriction G$ is an injection  and such that for every $x\in F^n\setminus\bigcup_{i\in k}\dom(p_i)$, and all indices $i, j\in k$, 
 $\diam(q(x))<\eps(\supp(p_i), x)$ and $\diam(q(x))\in b\setminus b_{ij}$ both hold. This is easy to do as the set $G$ countable, and for each point $x\in G$ the last two demands still leave infinitely many options for the value $q(x)$. It will be enough to show that $q$ is the desired common lower bound of the conditions $p_i$ for $i\in k$.

First of all, $q$ is a function and $\bigcup_ip_i\subset q$ by the definitions. To see that $q$ is a coloring, let $e=\{x_0, x_1, x_2\}$ be a $\Gamma$-hyperedge in $\dom(q)$; I must show that it is not monochromatic. There are several possible configurations. 

\noindent\textbf{Case 1.} $e\subset\dom(\bigcup_ip_i)$. This configuration is resolved in Claim~\ref{coloringclaim}.

\noindent\textbf{Case 2.} If $e$ contains exactly two vertices in $\dom(\bigcup_ip_i)$ and there is $i\in k$ such that these two vertices, say $x_1, x_2$ , both belong to $\dom(p_i)$. Then the hyperedge is not monochromatic by Proposition~\ref{epsproposition}(3) and $\diam(q(x))<\eps(\supp(p_i), x)$.

\noindent\textbf{Case 3.} If $e$ contains exactly two vertices in $\dom(\bigcup_ip_i)$ and there are distinct indices $i, j\in k$ such that (say) $x_1\in\dom(p_i\setminus \bar p)$ and $x_2\in\dom(p_j\setminus\bar p)$ then either $\{x_0, x_2\}\not\subset p_i(x_1)$ or $p_i(x_1)$ has diameter in $b_{ij}$ while $q(x_0)$ does not, by the choice of the set $b_{ij}(x)$. This means that the hyperedge $e$ is not monochromatic.

\noindent\textbf{Case 4.} If $e$ contains two or three vertices in the set $G$ then $e$ is not monochromatic since the function $q\restriction G$ is an injection.

Finally, I have to show that for all $i\in k$, $q\leq p_i$ holds. This is in fact another small ordeal. Fix $i\in k$; I must verify items (2) and (3) of Definition~\ref{posetdefinition}. Start with item (2). Let $x\in F^d$ be a point not in $\dom(p_i)$. There are two cases. 

\noindent\textbf{Case 1.} There is $j\neq i$ such that $x\in\dom(p_j)$. In this case, use Proposition~\ref{productproposition} and the mutual genericity assumption to see that $D(\mathbb{R}\cap V, x)=D(\supp(p_i), x)$ and then apply item (2) of Definition~\ref{posetdefinition} of $p_j\leq \bar p$ to see that $q(x)=p_j(x)$ has diameter smaller than $\eps(\supp(p_i), x)$.

\noindent\textbf{Case 2.} $x\notin\bigcup_j\dom(p_j)$; in this case, $\diam(q(x))<\eps(\supp(p_i), x)$ by the choice of the function $q$.

Now, to verify item (3), let $a\subset\mathbb{R}$ be an arbitrary finite set. For every index $j\neq i$, use Proposition~\ref{reflectionproposition} to find a finite set $b_j\subset\supp(p_j)$ such that for every point $x\in\dom(p_j)$, $D(\supp(p_i)\cup b_j, x)\subseteq D(\supp(p_i)\cup a, x)$. Use Proposition~\ref{productproposition} and the mutual genericity assumption to conclude that for every point $x\in\dom(p_j)$, $D((\mathbb{R}\cap V)\cup b_j, x)\subseteq D(\supp(p_i)\cup a, x)$. Conclude that $C(p_i, p_j, a)\cap\dom(p_j)\subseteq C(\bar p, p_j, b_j)$. Now, it is clear that $C(p_i, q, a)\subseteq G\cup\bigcup_{j\neq i}C(p_i, q, a)\cap\dom(p_j)\subseteq G\cup\bigcup_{j\neq i}C(\bar p, p_j, b_j)$. It follows that the set $q^*C(p_i, q, a)$ belongs to the ideal $I$: $q^*G\in I$ holds by the construction of the map $q$, and for every $j\in i$ $q^*C(p_i, q, a)\cap\dom(p_j)\subseteq q^*C(\bar p, p_j, b_j)\in I$ by Definition~\ref{posetdefinition}(3) applied to $p_j\leq\bar p$. The proof is complete.
\end{proof}

\begin{corollary}
\label{balancecorollary}
(CH) The poset $P$ is balanced and even $(4,3)$-balanced.
\end{corollary}

\begin{proof}
Let $p\in P$ be any condition. Proposition~\ref{barproposition} provides a total coloring $\bar p$ such that $\coll(\gw, \mathbb{R})\Vdash\bar p\leq p$. Proposition~\ref{balanceproposition} then shows that $\langle\coll(\gw, \mathbb{R}), \bar p\rangle$ is a (4,3)-balanced pair below $p$ as required by the definition of balance in \cite[Definition 13.1.1]{z:geometric}.
\end{proof}

\noindent In the very special case of the hypergraph of equilateral triangles, I get the following instrumental strengthening of balance to amalgamation diagrams with multiple generic extensions.

\begin{proposition}
\label{equibalanceproposition}
Let $d\geq 2$ be a number and let $\Gamma$ be the hypergraph of equilateral triangles on $\mathbb{R}^d$.
Let $j\in\gw$, and let $\langle V[G_i]\colon i\in j\rangle$ be a $j$-tuple of generic extensions such that

\begin{enumerate}
\item for distinct $i_0, i_1\in j$, $V[G_{i_0}]$ and $V[G_{i_1}]$ are mutually generic;
\item for pairwise distinct $i_0, i_1, i_2\in j$, $V[G_{i_0}, G_{i_1}]\cap V[G_{i_1}, G_{i_2}]\cap V[G_{i_0}, G_{i_2}]=V$.
\end{enumerate}

\noindent Assume that $\bar p\in V$ is a total coloring, and for each $i\in j$, $p_i\leq\bar p$ is a condition in $P$ as interpreted in the model $V[G_i]$. Then the conditions $\{p_i\colon i\in j\}$ have a common lower bound in $P$.
\end{proposition}

\noindent Note that mutually generic triples of generic extensions satisfy item (2) by the product forcing theorem, so the amalgamation assumptions here are weaker than in the case of Proposition~\ref{balanceproposition}. The weakening will be a critical ingredient in the proof of Theorem~\ref{maintheorem} below.

\begin{proof}
The proof is literally identical to the proof of Proposition~\ref{balanceproposition} except for the configuration in Case 3 of the proof  of Claim~\ref{coloringclaim}. To argue for that in the special case of equilateral triangles from the weaker assumptions on the tuple of generic extensions, proceed as follows.

Suppose that $\{x_0, x_1, x_2\}$ is an equilateral triangle in the domain of $\bigcup_{i\in k}p_i$ and there are pairwise distinct numbers $i_0, i_1, i_2\in j$ such that $x_k\in \dom(p_{i_j}\setminus p)$ for all $j\in 3$; re-numbering if necessary, we may assume that $i_j=j$ for all $j\in 3$. Let $\gd>0$ be the common length of the edges of the equilateral triangle $\{x_0, x_1, x_2\}$. Since $\gd\in V[G_0, G_1]\cap V[G_1, G_2]\cap V[G_0, G_2]$, the amalgamation assumption (2) implies that $\gd\in V$ holds. Now, the pairwise genericity assumption (1) and Proposition~\ref{productproposition} show that $D(V\cap\mathbb{R}, x_1)$ is a subset of both spheres of radius $\gd$ centered at $x_0$ and $x_2$, as these spheres are algebraic sets defined from parameters in $V[x_0]$ and $V[x_2]$ respectively and contain $x_1$. It follows that for every point $y\in D(V\cap\mathbb{R}, x_1)$, $\{x_0, y, x_2\}\in\Gamma$ holds. Now, Definition~\ref{posetdefinition}(2) applied to $p_1\leq p$ shows that the diameter of $p_1(x)$ is smaller than $\eps(V\cap\mathbb{R}, x_1)$. The definition of $\eps(V\cap\mathbb{R}, x_1)$ implies that $\{x_0, x_2\}\not\subset p_1(x_1)$; therefore, the triangle $\{x_0, x_1, x_2\}$ cannot be monochromatic.
\end{proof}

\section{The independence result}
\label{independencesection}

The proof of Theorem~\ref{maintheorem} relies on the methodology of geometric set theory. The key observation regarding isosceles triangles is the following.

\begin{proposition}
\label{isoproposition}
Let $\{x_0, x_1, x_2\}$ be an isosceles triangle in $\mathbb{R}^2$ Cohen-generic over the ground model in $V$. Then

\begin{enumerate}
\item each point $x_i$ for $i\in 3$ is a Cohen-generic point of $\mathbb{R}^2$;
\item the models $V[x_0], V[x_1], V[x_2]$ are pairwise mutually generic;
\item the intersection of the three models $V[x_0, x_1]$, $V[x_0, x_2]$, and $V[x_1, x_2]$ is equal to $V$.
\end{enumerate}
\end{proposition}

\noindent To parse the proposition correctly, consider the set $C\subset (\mathbb{R}^2)^3$ consisting of isosceles triangles in $\mathbb{R}^2$. This is a $G_\gd$-set, therefore Polish in the inherited topology. One can consider the Cohen poset of nonempty open subsets of $C$ ordered by inclusion; it adds a generic element of $C$. The triangle of Proposition~\ref{isoproposition} is generic for this Cohen poset.

\begin{proof}
The proof is a rather mechanical application of the calculus of generic points of semi-algebraic sets. For completeness, I recall the basic tenets of this calculus. If $X$ is a Polish space, a Cohen-generic element of $X$ over a transitive model $V$ of ZFC is a point $x\in X$ which belongs to all dense open subsets of $X$ coded in $V$. It is obtained by forcing with the Cohen poset $P_X$ of all nonempty open subsets of $X$ ordered by inclusion; its name for the generic point of $X$ will be denoted by $\dot x$. The following is proved by an immediate density argument.

\begin{claim}
\label{openclaim}
Let $X, Y$ be Polish spaces.

\begin{enumerate}
\item If $f\colon X\to Y$ is a continuous open function then $P_X$ forces the $f$-image of the generic point $\dot x$ to be $P_Y$-generic over $V$;
\item $P_{X\times Y}$ forces the points of generic pair $\langle \dot x, \dot y\rangle$ to be mutually generic over $V$ for the $P_X$ and $P_Y$ poset respectively.
\end{enumerate}
\end{claim}

\noindent If $m\geq 1$ is a number, then every semi-algebraic set $C\subset\mathbb{R}^m$ is a Boolean combination of closed sets by the quantifier elemination theorem, therefore it is $G_\gd$ in $\mathbb{R}^m$ and Polish in the inherited topology. In this sense I speak of Cohen-generic elements of $C$. The poset $P_C$ is realized as the poset of all basic open sets $O\subset\mathbb{R}^m$ with nonempty intersection with $C$, ordered by inclusion. 

Let $C\subset (\mathbb{R}^2)^3$ be the semi-algebraic set of all triples $\langle u_0, u_1, u_2\rangle$ consisting of pairwise distinct points such that the Euclidean distance from $u_0$ to $u_1$ or $u_2$ is the same. Denote the $P_C$-names for points in the generic triple by $\dot x_0, \dot x_1, \dot x_2$.

\begin{claim}
\label{3claim}
$P_C\Vdash\langle \dot x_0, \dot x_1\rangle$ is a Cohen generic point of $\mathbb{R}^2\times\mathbb{R}^2$. The same is true for any other pair of vertices of the generic triple.
\end{claim}

\begin{proof}
The projection from $C$ to any pair of coordinates is easily checked to be an open continuous map from $C$ to $\mathbb{R}^2\times\mathbb{R}^2$. Claim~\ref{openclaim} completes the proof.
\end{proof}

\noindent The following computation is a little more involved. The proof is a great example of a powerful method I call the \emph{duplication technique}.

\begin{claim}
\label{1claim}
$P_C$ forces $V[\dot x_0, \dot x_1]\cap V[\dot x_1, \dot x_2]=V[\dot x_1]$.
\end{claim}

\begin{proof}
Let $D\subset(\mathbb{R}^2)^4$ be the semi-algebraic set of all quadruples $\langle u_0, u_{00}, u_1, u_2\rangle$ consisting of pairwise distinct points such that $\langle u_0, u_1, u_2\rangle\in C$ and $\langle u_{00}, u_1, u_2\rangle\in C$. It is easily checked that the two projections to the coordinates indexed by $0, 1, 2$ and to the coordinates indexed by $00, 1, 2$ are both continuous open maps from $D$ to $C$, and the projection to the coordinates indexed by $0, 00, 1$ is a continuous open function from $D$ to $(\mathbb{R}^2)^3$. Denote the $P_D$-names for the points in the generic quadruple by $\dot x_0, \dot x_{00}, \dot x_1, \dot x_2$. Claim~\ref{openclaim} then immediately shows that $P_D$ forces the following:

\begin{itemize}
\item $\langle \dot x_0, \dot x_1, \dot x_2 \rangle$ is a Cohen-generic point of $C$;
\item $\langle \dot x_{00}, \dot x_1, \dot x_2 \rangle$ is a Cohen-generic point of $C$;
\item $\langle \dot x_0, \dot x_{00}, \dot x_1 \rangle$ is a Cohen-generic point of $(\mathbb{R}^2)^3$.
\end{itemize}

 \noindent To prove the claim, suppose that $O_0\times O_1\times O_2$ is a condition in the poset $P_C$, and $\tau_{01}, \tau_{02}$ are two P$_{\mathbb{R}^2\times\mathbb{R}^2}$-names for sets of ordinals such that $O_0\times O_1\times O_2$ forces that $\tau_{01}/\dot x_0,\dot x_1=\tau_{02}/\dot x_1, \dot x_2$. I must find a stronger condition which forces the common value to belong to the model $V[\dot x_2]$. To this end, consider the condition $p=O_0\times O_0\times O_1\times O_2$ in $P_D$. The first two items above imply that this condition forces $\tau_{01}/\dot x_0, \dot x_1$ and $\tau_{01}/\dot x_{00}, \dot x_1$ to be both equal to the value $\tau_{12}/\dot x_1, \dot x_2$; in particular, they are forced to be equal. Their common value belongs to the intersection $V[\dot x_0, \dot x_1]\cap V[\dot x'_0, \dot x_1]$. However, the third item above shows that these two models are mutually generic extensions of $V[\dot x_1]$; so, the common value must belong to the model $V[\dot x_1]$ by the product forcing theorem. It follows that some condition below $O_0\times O_1\times O_2$ must force $\tau_{01}/\dot x_0, \dot x_1\in V[\dot x_1]$ as desired.
\end{proof}

\noindent A symmetrical argument shows the following:

\begin{claim}
\label{2claim}
$P_C$ forces $V[\dot x_0, \dot x_2]\cap V[\dot x_1, \dot x_2]=V[\dot x_2]$.
\end{claim}

\noindent Now, Claims~\ref{1claim} and~\ref{2claim} show that the intersection of the models $V[\dot x_0, \dot x_1]$, $V[\dot x_1, \dot x_2]$ and $V[\dot x_0, \dot x_2]$ is forced to be equal to the intersection of $V[\dot x_0]$ and $V[\dot x_1]$. Claim~\ref{3claim} shows that these two extensions are mutually generic extensions of the ground model, showing that their intersection is $V$ and completing the proof of the proposition.
\end{proof}

The rest of the proof of Theorem~\ref{maintheorem} is now routine and follows the lines of previous arguments in \cite{z:geometric}. Let $d\geq 2$ be a number. Let $\kappa$ be an inaccessible cardinal. Let $W$ be the choiceless Solovay model derived from the cardinal $\kappa$ \cite[Theorem 26.14]{jech:newset}; by this I mean the subclass of the $\coll(\gw, <\kappa)$-extension of the ground model $V$ consisting of sets hereditarily definable from elements of $V$ and infinite binary sequences in the extension. Let $P$ be the coloring poset of Definition~\ref{posetdefinition} for the hypergraph of equilateral triangles on $\mathbb{R}^d$. Let $G\subset P$ be a filter generic over the model $W$. It will be enough to show that in $W[G]$, DC holds, $\Gamma$ has countable chromatic number, and every non-meager subset of $\mathbb{R}^2$ contains an isosceles triangle.

To argue for these points, first argue that $P$ is $\gs$-closed (Proposition~\ref{sigmaclosureproposition}), $\gs$-closed posets preserve DC; thus, as DC holds in $W$, it holds in $W[G]$ as well. Since $P$ is $\gs$-closed, the extension $W[G]$ does not contain any new reals; consequently, $\bigcup G$ is a $\Gamma$-coloring by the density Proposition~\ref{densityproposition}. Finally, we must show that every non-meager subset of $\mathbb{R}^2$ contains an isosceles triangle. To this end, return to $W$, let $p\in P$ be a condition, and let $\tau$ be a name for a non-meager subset of $\mathbb{R}^2$. Let $z\in\cantor$ be a parameter such that $\tau$ is definable from $\tau$ and an element of $V$. Let $V[K]$ be a model intermediate between $V$ and $W$, containing $z$ and $p$ satisfying the Continuum Hypothesis, and obtained by a poset of cardinality smaller than $\kappa$. 

In the model $V[K]$, find a total $\Gamma$-coloring $\bar p$ such that $p\subset\bar p$ and $\coll(\gw, \mathbb{R})\Vdash\bar p\leq p$ (Proposition~\ref{barproposition}). Since $\tau$ is a $P$-name for a nonmeager set, back in $W$ it is true that $\bar p\Vdash\tau$ contains a point Cohen-generic over $V[K]$. Let $Q$ be the Cohen poset on $\mathbb{R}^2$.  By the forcing theorem applied in $V[K]$ and standard homogeneity facts about $W$, there must be a basic open subset $O\subset\mathbb{R}^2$, a cardinal $\lambda\in\kappa$, and a $Q\times\coll(\gw, \lambda)$-name $\gs$ for a condition in $P$ stronger than $\bar p$ such that

$$O\Vdash_Q\coll(\gw, \lambda)\Vdash\coll(\gw, <\kappa)\Vdash\gs\Vdash_P\dot x\in\tau$$

\noindent where $\dot x$ is the $Q$-name for the Cohen-generic point of $\mathbb{R}^2$ added by $Q$ and $\tau$ is replaced by its definition from the parameter $z$.

Let $R$ be the poset for adding a Cohen-generic isosceles triangle in $\mathbb{R}^2$. Back in the model $W$, pick a triple $\langle x_i\colon i\in 3\rangle$ $R$-generic over $V[K]$ and meeting the condition $O^3\in R$. Proposition~\ref{isoproposition} applied in $V[K]$ shows that the models $V[K][x_i]$ for $i\in 3$ are pairwise mutually generic over $V[K]$, and $\bigcap_{a\in [3]^2}V[K][x_i\colon i\in a]=V[K]$. Pick filters $H_i\subset\coll(\gw, \lambda)$ for $i\in 3$ which are mutually generic over $V[K][x_i\colon i\in 3]$.

\begin{claim}
The models $V[K][x_i][H_i]$ for $i\in 3$ are pairwise mutually generic extensions of $V[K]$. In addition, $\bigcap_{a\in [3]^2}V[K][x_i\colon i\in a][H_i\colon i\in a]=V[K]$.
\end{claim}

\begin{proof}
The first sentence is immediate. For the second sentence, I will show that $V[K][x_0, x_1][H_0]\cap V[K][x_0, x_2][H_0]\cap V[K][x_1, x_2][H_0]=V[K]$; the claim is then obtained by repeating this argument three times.

Suppose that $b$ is a set of ordinals in the intersection of the three models. Since $H_0$ is generic over $V[K][x_0, x_1, x_2]$ and $b\in V[K][x_1, x_2]$ holds, a genericity argument with $H$ in the models $V[K][x_0, x_1]$ and $V[K][x_0, x_2]$ shows that $b$ must be in both of these models. It follows that $b\in \bigcap_{a\in [3]^2}V[K][x_i\colon i\in a]$ and therefore in $V[K]$ as desired.
\end{proof}

\noindent By the balance Proposition~\ref{equibalanceproposition}, the conditions $\gs/x_i, H_i$ for $i\in 3$ have a common lower bound in the poset $P$. That lower bound forced the isosceles triangle $\langle x_i\colon i\in 3\rangle$ to be a subset of $\tau$. Thus completes the proof of Theorem~\ref{maintheorem}.

\bibliographystyle{plain} 
\bibliography{odkazy,zapletal,shelah}
\end{document}